\documentclass{article}

\usepackage{graphicx} 
\usepackage{subfigure} 

\usepackage{amsmath}

\usepackage{fourier}
\usepackage[T1]{fontenc}

\usepackage{braket}

\DeclareMathOperator{\tr}{Tr}

\DeclareMathOperator{\dom}{dom}
\DeclareMathOperator{\inte}{int}

\DeclareMathOperator*{\argmin}{\arg \min}

\newcommand{\rhostar}{\rho^\star}
\newcommand{\norm}[1]{\Vert #1 \Vert}

\newcommand{\du}{\mathrm{d}}

\usepackage[standard,amsmath,thmmarks]{ntheorem}

\usepackage{algorithm}
\usepackage{algpseudocode}

\usepackage{hyperref}

\title{A General Convergence Result for the Exponentiated Gradient Method}
\author{Yen-Huan~Li and Volkan~Cevher \\[5pt] \small Laboratory for Information and Inference Systems \\[-3pt] \small \'{E}cole Polytechnique F\'{e}d\'{e}rale de Lausanne \\[-3pt] \small CH-1015 Lausanne, Switzerland }
\date{\empty}

\begin{document} 

%
%
%

\maketitle

\begin{abstract} 
The batch exponentiated gradient (EG) method provides a principled approach to convex smooth minimization on the probability simplex or the space of quantum density matrices.
However, it is not always guaranteed to converge.
Existing convergence analyses of the EG method require certain quantitative smoothness conditions on the loss function, e.g., Lipschitz continuity of the loss function or its gradient, but those conditions may not hold in important applications.
In this paper, we prove that the EG method with Armijo line search always converges for \emph{any} convex loss function with a locally Lipschitz continuous gradient. 
Because of our convergence guarantee, the EG method with Armijo line search becomes the fastest guaranteed-to-converge algorithm for maximum-likelihood quantum state estimation, on the real datasets we have.
\end{abstract} 

\section{Introduction} \label{sec_intro}

\subsection{Problem Formulation} \label{sec_problem}

Consider the convex minimization problem
\begin{equation}
\rhostar \in \argmin \Set{ f ( \rho ) | \rho \in \mathcal{D} } , \label{eq_prob}
\end{equation}
where $f$ is a continuously differentiable convex loss function, and $\mathcal{D}$ is the set of \emph{(quantum) density matrices}, i.e., for some $d \in \mathbb{N}$, 
\begin{equation}
\mathcal{D} := \Set{ \rho \in \mathcal{C}^{d \times d} | \rho = \rho^H, \rho \geq 0, \tr ( \rho ) = 1 } . \notag
\end{equation}
A density matrix is a non-commutative analog of a probability distribution---if $\rho$ is diagonal, its diagonal elements define a probability distribution on $\set{ 1, \ldots, d }$.

The (batch) exponentiated gradient (EG) method \cite{Arora2007,Kivinen1997,Tsuda2005} provides a principled approach to solving such a convex program.
Starting with some non-singular density matrix $\rho_0$, the EG method iterates as
\begin{equation}
\rho_{k + 1} = c_k^{-1} \exp \left[ \log ( \rho_k ) - \alpha_k f' ( \rho_k ) \right] , \quad k \in \mathbb{Z}_+ , \label{eq_EG}
\end{equation}
for some given step size $\alpha_k$, where $c_k$ is a positive number normalizing the trace of $\rho_{k + 1}$.
The EG method, in its formulation, is also a special case of mirror descent \cite{Beck2003,Nemirovsky1983} and the interior gradient method \cite{Auslender2006}. 
We choose to call \eqref{eq_EG} the EG method, as this name refers exactly to the expression we consider.

Our goal is to show that if the step sizes are computed by Armijo line search, the EG method converges for \emph{almost all} continuously differentiable convex loss functions.
We will define precisely the class of loss functions we consider in Section \ref{sec_contribution}.  

By considering only diagonal matrices, the convex program \eqref{eq_prob} and the EG method \eqref{eq_EG} are equivalent to their vector counterparts, respectively (see, e.g., Section 4.3 in \cite{Bubeck2015a} for the vector formulation).
The theory in this paper hence automatically specializes to the vector case. 

\subsection{Motivation} \label{sec_motivation}

To derive a step size $\alpha_k$ that guarantees the convergence rate of the EG method, one needs to impose some quantitative smoothness condition on the loss function.
The standard condition is $L$-Lipschitz continuity of the loss function or its gradient on $\mathcal{D}$ \cite{Auslender2006,Beck2003,Nemirovsky1983}. 
$L$-Lipschitz continuity with respect to the relative entropy, instead of a norm, was considered in \cite{Bartlett2007,Collins2008}.
An $L$-Lipschitz-like condition was proposed in \cite{Bauschke2016}, requiring $L h - f$ to be convex for some $L > 0$, where $h$ denotes the negative entropy function.
The Lipschitz-like condition was later shown to be equivalent to $L$-Lipschitz continuity of the gradient with respect to the relative entropy in \cite{Lu2016}. 
Once a condition is verified and the corresponding parameter $L$ is explicitly computed, the step size $\alpha_k$ is then set as a function of $L$ and the iteration counter $k$.

However, the conditions may not hold, and verifying the conditions is usually non-trivial.
For instance, consider minimizing the loss function
\begin{equation}
f_1 ( x, y ) := - \log ( x ) - \log ( y ) , \notag
\end{equation}
on the probability simplex 
\begin{equation}
\mathcal{P} := \Set{ ( x, y ) \in \mathbb{R}^2 | x \geq 0, y \geq 0, x + y = 1 } . \notag 
\end{equation}
Neither $f_1$ nor its gradient $f_1'$ is Lipschitz continuous, due to the presence of the logarithmic function.
The Lipschitz-like condition \cite{Bartlett2007,Bauschke2016,Collins2008,Lu2016} requires the convexity of $L h ( x, y ) - f_1 ( x, y )$ for some $L > 0$ on $\mathcal{D}$, where $h ( x, y )$ is the negative entropy function: 
\begin{equation}
h ( x, y ) := x \log ( x ) + y \log ( y ) . \notag
\end{equation}
A necessary condition is
\begin{equation}
L \frac{\partial^2 h ( x, y )}{\partial x^2} - \frac{\partial^2 f_1 ( x, y )}{ \partial x ^ 2 } = \frac{L}{x} - \frac{1}{x^2} \geq 0 , \quad \text{for all } ( x, y ) \in \mathcal{P}, \notag
\end{equation}
which cannot hold for any fixed $L$, because $x$ can be arbitrarily close to zero.

The loss function $f_1$ is not simply an artificial example.
Consider a generalization of minimizing 
\begin{equation}
f_2 ( x ) := - \frac{1}{n} \sum_{i = 1}^n \log \braket{ b_i, x } \notag
\end{equation}
on the probability simplex for some $n \in \mathbb{N}$, where $b_1, \ldots, b_n$ are vectors in the non-negative orthant, for which $f_1$ is a special case with $b_1 = ( 1, 0 )$ and $b_2 = ( 0, 1 )$.
A minimizer corresponds to the best constant rebalanced strategy for log-optimal portfolio selection \cite{Cover1991}.
Consider a further generalization under the non-commutative setting: 
\begin{equation}
f_3 ( \rho ) := - \frac{1}{n} \sum_{i = 1}^n \log \tr ( M_i \rho ) , \quad \rho \in \mathcal{D} , \label{eq_f2}
\end{equation}
where $M_1, \ldots, M_n$ are given positive semi-definite matrices in $\mathbb{C}^{d \times d}$.
A minimizer of $f_3$ on $\mathcal{D}$ is a maximum-likelihood (ML) estimate for quantum state estimation \cite{Hradil1997}, and also an ML estimate of the PhaseLifted signal for phase retrieval with Poisson noise \cite{Odor2016}. 

As log-optimal portfolio selection by the EG method had been studied under the on-line setting (see, e.g., \cite{Cesa-Bianchi2006,Helmbold1998}), it is possible to extend existing results to the batch non-commutative formulation (i.e., minimizing $f_2$ on $\mathcal{D}$).
Such an extension, however, might not be able to address all other cases.
For example, the hedged approach to ML quantum state estimation considers minimizing $f_2 - \lambda_1 \log \det ( \rho )$ for some $\lambda_1 > 0$ \cite{Blume-Kohout2010}; the max-entropy approach considers minimizing $f_2 + \lambda_2 \tr ( \rho \log \rho )$ for some $\lambda_2 > 0$ \cite{Teo2011}; the approach to low-rank matrix estimation proposed in \cite{Koltchinskii2011b} considers minimizing $\sum_i \left[ y_i - \tr ( M_i \rho ) \right]^2 + \lambda_3 \tr ( \rho \log \rho )$ for some real numbers $y_i$, Hermitian matrices $M_i$, and $\lambda_3 > 0$; and a similar vector formulation of empirical risk minimization with Shannon entropy penalization was studied in \cite{Koltchinskii2009a}.
In all examples, the loss functions are not Lipschitz continuous in function values nor their gradients. 

Why do we not use the projected gradient method? 
Indeed, it was shown in \cite{Gafni1982} that the projected gradient method with Armijo line search converges for minimizing any continuously differentiable loss function. 
We notice that, however, the projected gradient method may be not well-defined. 
Consider minimizing $f_1$ on the probability simplex as an example.
As projection onto the probability simplex often results in a sparse output, it can happen that some iterate $( x_k, y_k )$ is exactly sparse; then $f_1 ( x_k, y_k )$ and $f_1' ( x_k, y_k )$ are not defined, and the algorithm is forced to terminate.
An explicit example is given by setting $( x_{k - 1}, y_{k - 1} ) = ( 0.99999, 0.00001 )$ and the step size (or the upper bound of it for Armijo line search) to be $1$, for which $( x_k, y_k ) = ( 0, 1 )$.



\subsection{Our Contribution} \label{sec_contribution}

Unlike existing results, we are interested in seeking for an \emph{universal} approach to convex smooth minimization on $\mathcal{D}$, which converges for minimizing almost all continuously differentiable convex functions.

We consider finding the step sizes by the Armijo line search rule.
The pseudo code is shown in Algorithm \ref{alg}, in which we define
\begin{equation}
\rho ( \alpha ) := c^{-1} \exp \left[ \log ( \rho ) - \alpha f' ( \rho ) \right], \label{eq_rho_alpha}
\end{equation}
for any non-singular density matrix $\rho$ and $\alpha > 0$, where the positive number $c$ normalizes the trace of $\rho ( \alpha )$.
The outer for-loop in Algorithm \ref{alg} implements the EG method; the inner while-loop applies the Armijo rule to find a proper step size.

\begin{algorithm}

\caption{Exponentiated Gradient Method with Armijo Line Search}

\label{alg}

\begin{algorithmic}[1]
\Require 
$\alpha > 0$, $r \in ( 0, 1 )$, $\tau \in ( 0, 1 )$, $\rho_0 \in \inte ( \mathcal{D} )$
\For{$k = 0, 1, 2, \ldots$}
\State $\alpha_k \leftarrow \alpha$
\While{$f( \rho_{k} ( \alpha_k ) ) > f ( \rho_{k} ) + \tau \Braket{ f' ( \rho_k ), \rho_k ( \alpha_k ) - \rho_k } $}
\State $\alpha_k \leftarrow r \alpha_k$
\EndWhile
\State $\rho_{k + 1} \leftarrow \rho_k ( \alpha_k )$
\EndFor
\end{algorithmic}

\end{algorithm}

The EG method with Armijo line search had been studied in \cite{Auslender2004,Auslender2006}, but the analyses therein assume Lipschitz continuity of $f'$.

Our contribution lies in deriving a convergence guarantee under a very weak smoothness condition on the loss function.

\begin{definition}
We say that $f$ has a locally Lipschitz continuous gradient, if for every $x \in \dom ( f )$, there exists a neighborhood in $\dom ( f )$ on which $f'$ is Lipschitz continuous.
\end{definition}

\begin{remark}
It is easily checked that if $f$ is twice continuously differentiable on $\dom ( f )$, then $f$ has a locally Lipschitz continuous gradient.
Therefore, for instance, the functions $f_1$,  $f_2$, and $f_3$ all have locally Lipschitz continuous gradients.
\end{remark}

The main result of this paper is Theorem \ref{thm_main}, which is proved in Section \ref{sec_proof}.

\begin{theorem} \label{thm_main}
Consider solving the convex program \eqref{eq_prob} by Algorithm \ref{alg}.
Assume that $f$ has a locally Lipschitz continuous gradient, and $\dom ( f )$ contains all non-singular density matrices.
The following statements hold.
\begin{enumerate}
\item The Armijo line search (Line 3--5) terminates in finite steps.
\item $\rho_k \in \mathcal{D}$ for all $k$.
\item $f ( \rho_{k + 1} ) \leq f ( \rho_k )$ for all $k$.
\item The sequence $( \rho_k )_{k \in \mathbb{N}}$ has at least one limit point.
\item Every limit point of $( \rho_k )_{k \in \mathbb{N}}$ minimizes $f$ on $\mathcal{D}$.
\end{enumerate}
\end{theorem}

Notice that both Algorithm \ref{alg} and Theorem \ref{thm_main} do not assume the local Lipschitz constants of $f'$ to be known nor uniformly bounded.

Our problem formulation does not impose any quantitative smoothness condition on the loss function, so we do not have a guarantee on the convergence rate. 
Numerical experiments on ML quantum state estimation (Section \ref{sec_exp}), nevertheless, show that the empirical convergence rate of the EG method with Armijo line search can be competitive.
In fact, the EG method with Armijo line search is \emph{the fastest} among all existing guaranteed-to-converge algorithms for ML quantum state estimation, on the real experimental data we have.
Recall that existing analyses for the EG method, with and without line search, do not directly apply to ML quantum state estimation, and the projected gradient method is, rigorously speaking, not applicable.

\subsection{Notations}

Let $g$ be a convex function taking values in $\mathbb{R} \cup \Set{ \pm \infty }$. 
The (effective) domain of $g$, denoted by $\dom ( g )$, is given by $\dom ( g ) = \Set{ x | g ( x ) < + \infty }$.
We denote the gradient of $g$ by $g'$, and the Hessian by $g''$. 

We will focus on the non-commutative formulation \eqref{eq_prob} in the rest of this paper.
To define the gradient of $f$ properly is tricky, as a non-constant real-valued function of complex variables cannot be analytic.
We define $f' ( x )$ at $x \in \dom ( f )$ as the unique matrix such that 
\begin{equation}
f ( y ) \geq f ( x ) + \Braket{ f' ( x ), y - x } , \notag
\end{equation}
for all $y \in \dom ( f )$, where the inner product is the Hilbert-Schmidt inner product, i.e., for any matrices $X, Y \in \mathbb{C}^{d \times d}$, 
\begin{equation}
\Braket{ X, Y } := \tr ( X^H Y ). \notag
\end{equation}
The definition of the EG method (cf. \eqref{eq_EG}) presumes that $f'$ is Hermitian.

The inner products in the rest of this paper will be all Hilbert-Schmidt, unless otherwise specified.
We denote by $\norm{ \cdot }_F$ the Frobenius norm, and $\norm{ \cdot }_{\tr}$ the trace norm.

The functions $\exp ( \cdot )$ and $\log ( \cdot )$ in \eqref{eq_EG} are matrix exponential and matrix logarithmic functions.
Generally speaking, let $X = \sum_{j \in \mathcal{J}} \lambda_j P_j$ be the spectral decomposition of a Hermitian matrix $X$, where $P_j$ is the projection onto the eigenspace corresponding to $\lambda_j$ for all $j \in \mathcal{J}$.
Let $g$ be a real-valued function whose domain contains $\set{ \lambda_j : j \in \mathcal{J} }$.
Then $g ( X )$ is defined as $\sum_{j \in \mathcal{J}} g ( \lambda_j ) P_j$.

The von Neumann entropy of a density matrix $\rho$ is given by
\begin{equation}
h ( \rho ) := - \tr ( \rho \log \rho ) , \notag
\end{equation}
where we adopt the convention that $0 \log 0 = 0$.
The quantum relative entropy between two density matrices $\rho$ and $\sigma$, denoted by $D ( \rho, \sigma )$, is given by
\begin{equation}
D ( \rho, \sigma ) := \left\{ \begin{array}{ll} 
\tr ( \rho \log \rho ) - \tr ( \rho \log \sigma ) & \text{if } \ker ( \rho ) \supseteq \ker ( \sigma ) , \\
+ \infty & \text{otherwise} .
\end{array} \right. \notag
\end{equation}
The relative entropy is always non-negative.
Two non-singular density matrices $\rho$ and $\sigma$ are the same, if and only if $D ( \rho, \sigma ) = 0$.

%

\section{Proof of Theorem \ref{thm_main}} \label{sec_proof}

Section \ref{sec_prelim} provides some necessary background knowledge. 
Section \ref{sec_PB} presents a local Peierls-Bogoliubov inequality, which is key in establishing the convergence statement in Theorem \ref{thm_main}.
Section \ref{sec_main_proof} shows the complete proof of Theorem \ref{thm_main}.

\subsection{Preliminaries} \label{sec_prelim}

We defined $\rho ( \alpha )$ explicitly in \eqref{eq_rho_alpha}. 
The following lemma shows that $\rho ( \alpha )$ admits an equivalent definition. 

\begin{lemma} \label{lem_equiv}
For any non-singular density matrix $\rho$ and $\alpha > 0$, one has
\begin{equation}
\rho ( \alpha ) = \argmin \Set{ \Braket{ f' ( \rho ), \sigma - \rho } + \frac{1}{\alpha} D ( \sigma, \rho ) | \sigma \in \mathcal{D} } . \label{eq_equiv}
\end{equation}
\end{lemma}

\begin{proof}
Combine the arguments in \cite{Beck2003} and Section 4.3 of \cite{Bubeck2015a}, or directly solve the convex program as in \cite{Tsuda2005}.
\end{proof}

Notice that $\rho$ itself is a feasible point of the convex program \eqref{eq_equiv}. 
One then has
\begin{equation}
\Braket{ f' ( \rho ), \rho ( \alpha ) - \rho } + \frac{1}{\alpha} D ( \rho ( \alpha ), \rho ) \leq 0 . \notag
\end{equation}
This proves the following corollary.

\begin{corollary} \label{cor_equiv}
For any non-singular density matrix $\rho$ and $\alpha > 0$, one has
\begin{equation}
\Braket{ f' ( \rho ), \rho ( \alpha ) - \rho } \leq - \frac{D ( \rho ( \alpha ), \rho )}{ \alpha } . \notag
\end{equation}
\end{corollary}

Lemma \ref{lem_equiv} implies a fixed-point characterization of a minimizer.

\begin{lemma} \label{lem_fp}
A non-singular density matrix $\rho$ minimizes $f$ on $\mathcal{D}$, if $\rho = \rho ( \alpha )$ for some $\alpha > 0$. 
On the other hand, if a non-singular density matrix $\rho$ minimizes $f$ on $\mathcal{D}$, then $\rho = \rho ( \alpha )$ for all $\alpha \geq 0$.
\end{lemma}

\begin{proof}
The first-order optimality condition (see, e.g., \cite{Nesterov2004}) says that $\rho$ is a minimizer, if and only if
\begin{equation}
\Braket{ f' ( \rho ), \sigma - \rho } \geq 0 , \notag
\end{equation}
for all $\sigma \in \mathcal{D}$.
Equivalently, we write
\begin{equation}
\Braket{ f' ( \rho ) + \alpha^{-1} \tilde{h}' ( \rho ) - \alpha^{-1} \tilde{h}' ( \rho ), \sigma - \rho } \geq 0 , \label{eq_redundant_optimality}
\end{equation}
where $\tilde{h} ( \rho ) := \tr ( \rho \log \rho ) - \tr ( \rho )$. 
It is easily checked that \eqref{eq_redundant_optimality} is the optimality condition of
\begin{equation}
\rho = \argmin \Set{ \Braket{ f' ( \rho ), \sigma - \rho } + \alpha^{-1} D ( \sigma, \rho ) | \sigma \in \mathcal{D} } ,  \notag
\end{equation}
as $D ( \cdot, \cdot )$ coincides with the Bregman divergence defined by $\tilde{h}$ on $\mathcal{D} \times \mathcal{D}$ (see, e.g., \cite{Petz2008}).
The lemma then follows from Lemma \ref{lem_equiv}.
\end{proof}

The local Lipschitz continuity of $f'$ allows us to bound the first-order approximation error \emph{locally}. 

\begin{lemma} \label{lem_epsilon}
Let $\rho$ be a non-singular density matrix.
For $\alpha$ small enough, one has
\begin{align}
0 \leq f ( \rho ( \alpha ) ) - \left[ f ( \rho ) + \Braket{ f' ( \rho ), \rho ( \alpha ) - \rho } \right] \leq L_\rho D ( \rho ( \alpha ), \rho ) , \notag
\end{align}
where $L_\alpha$ is the local Lipschitz continuity constant for $f'$ in a neighborhood of $\rho$. 
\end{lemma}

\begin{proof}
Notice that $\rho ( \alpha )$ is a continuous function wrt $\alpha$.
Following the proof of Lemma 1.2.3 in \cite{Nesterov2004}), one has
\begin{equation}
0 \leq f ( \rho ( \alpha ) ) - \left[ f ( \rho ) + \Braket{ f' ( \rho ), \rho ( \alpha ) - \rho } \right] \leq \frac{L_\rho}{2} \norm{ \rho ( \alpha ) - \rho }_F^2 , \notag
\end{equation}
for small enough $\alpha$, where $L_\rho$ denotes the local Lipschitz constant of $f'$.
By Pinsker's inequality \cite{Hiai1981}, one has
\begin{equation}
\frac{L_\rho}{2} \norm{ \rho ( \alpha ) - \rho }_F^2 \leq \frac{L_\rho}{2} \norm{ \rho ( \alpha ) - \rho }_{\tr}^2 \leq L_\rho D ( \rho ( \alpha ), \rho ) , \notag
\end{equation}
which proves the lemma.
\end{proof}

%
%

\subsection{A Local Peierls-Bogoliubov Inequality} \label{sec_PB}

Let $\rho$ be any non-singular density matrix.
Define
\begin{equation}
\varphi ( \alpha; \rho ) := \log \tr \exp \left[ \log ( \rho ) - \alpha f' ( \rho ) \right]. \notag
\end{equation}
The function $\varphi$ plays a key role in the proof of Theorem \ref{thm_main}.
We will often omit $\rho$ and write $\varphi ( \alpha )$ for convenience, when the corresponding $\rho$ is irrelevant, or clear from the context.


The Peierls-Bogoliubov inequality says that $\varphi$ is a convex function (see, e.g., \cite{Carlen2010}); equivalently, one has $\varphi'' ( \alpha ) \geq 0$ for all $\alpha \in \mathbb{R}$.
In this paper, we need a slightly stronger version. 

\begin{theorem}[Peierls-Bogoliubov Inequality] \label{thm_PB}
One has $\varphi'' ( \alpha ) \geq 0$ for all $\alpha \in \mathbb{R}$. 
Moreover, $\varphi'' ( \alpha ) = 0$, if and only if $f' ( \rho ) = \kappa I$ for some $\kappa \in \mathbb{R}$.
\end{theorem}

\begin{proof}
The proof below is essentially a combination of the proofs in \cite{Ohya1993} and \cite{Ohya2011}. 
We show it to identify the condition for $\varphi'' = 0$.

Let $A$, $B$ be two Hermitian matrices.
Define $H_t := A + t B$, and $\Phi ( t ) := \log \tr \exp ( H_t )$ for $t \in \mathbb{R}$. 
By the relation \cite{Wilcox1967} 
\begin{equation}
\frac{\partial \exp ( H_t )}{ \partial t } = \int_0^1 \exp \left[ ( 1 - u ) H_t \right] B \exp \left( u H_t \right) \, \du u , \notag
\end{equation}
one can obtain
\begin{equation}
\Phi'' ( t ) = \frac{ \Braket{ B, B }_{\text{BKM}} \Braket{ I, I }_{\text{BKM}} - \Braket{ I, B }_{\text{BKM}}^2 }{ \left[ \tr \exp ( H_t ) \right] ^ 2 }, \notag
\end{equation}
where $\braket{ \cdot, \cdot }_{\text{BKM}}$ denotes the Bogoliubov-Kubo-Mori inner product with respect to $H_t$: 
\begin{equation}
\Braket{ X, Y }_{\text{BKM}} := \int_0^1 \tr \left\{ \exp \left[ ( 1 - u ) H_t \right] X \exp ( u H_t ) Y \right\} \, \du u , \notag
\end{equation}
for any Hermitian matrices $X, Y$.
Set $A = \log ( \rho )$ and $B = - f' ( \rho )$. 
The theorem follows from the Cauchy-Schwarz inequality and its equality condition.
\end{proof}

%

%
%

The following lemma establishes the connection between $\varphi$ and the EG method, which is easy to prove, but perhaps not obvious at first glance.

\begin{lemma} \label{lem_D_phi}
For any non-singular density matrix $\rho$ and $\alpha > 0$, one has
\begin{align}
D ( \rho ( \alpha ), \rho ) &= \varphi ( 0 ) - \left[ \varphi ( \alpha ) + \varphi' ( \alpha ) ( 0 - \alpha ) \right], \notag \\
D ( \rho, \rho ( \alpha ) ) &= \varphi ( \alpha ) - \left[ \varphi ( 0 ) + \varphi' ( 0 ) ( \alpha - 0 ) \right] . \notag
\end{align}
\end{lemma}

\begin{proof}
By Theorem 3.23 in \cite{Hiai2014}, one can obtain
\begin{equation}
\varphi' ( \alpha ) = \frac{ \tr \left\{ - f' ( \rho ) \exp \left[ \log ( \rho ) - \alpha f' ( \rho ) \right] \right\}  }{ \tr \exp \left[ \log ( \rho ) - \alpha f' ( \rho ) \right] } . \notag
\end{equation}
The lemma is then verified by direct calculation. 
\end{proof}

We now prove the main result of this sub-section, a \emph{local} Peierls-Bogoliubov inequality.
Its formulation was motivated by a result in \cite{Gafni1982}, which, in the context of this paper, says that the mapping
\begin{equation}
\alpha \mapsto \frac{\norm{ \Pi_{\mathcal{D}} ( \rho - \alpha f' ( \rho ) ) }_F}{ \alpha } \notag
\end{equation}
is non-increasing on $( 0, + \infty )$, where $\Pi_{\mathcal{D}}$ denotes the projection onto $\mathcal{D}$ with respect to the Forbenius norm $\norm{ \cdot }_F$.

\begin{proposition}[Local Peierls-Bogoliubov Inequality] \label{prop_local_PB}
For any non-singular density matrix $\rho$ and $\bar{\alpha} > 0$, there exists some $\gamma \geq 2$ such that 
\begin{equation}
\Gamma ( \alpha ) := \frac{ D ( \rho ( \alpha ), \rho ) }{ \alpha^\gamma } \label{eq_Bertsekas}
\end{equation}
is non-increasing on $( 0, \bar{\alpha} ]$.
Moreover, $\gamma$ depends continuously on $\rho$.
\end{proposition}


\begin{proof}
We prove the proposition by verifying $\Gamma' ( \alpha ) \leq 0$ on $( 0, \bar{\alpha} ]$.
Applying Lemma \ref{lem_D_phi}, a direct calculation gives
\begin{align}
\Gamma' ( \alpha ) & = \frac{\varphi ( \alpha ) - \varphi' ( \alpha ) \alpha + \frac{\varphi'' ( \alpha )}{\gamma} \alpha ^ 2}{ \gamma^{-1} \alpha^{\gamma + 1} } \notag \\
& = \frac{\varphi ( \alpha ) + \varphi' ( \alpha ) ( 0 - \alpha ) + \frac{\varphi'' ( \alpha )}{\gamma} ( 0 - \alpha ) ^ 2}{ \gamma^{-1} \alpha^{\gamma + 1} } . \notag
\end{align}
Notice that $0 = \varphi ( 0 )$. 
Then one has $\Gamma' ( \alpha ) \leq 0$, if and only if
\begin{equation}
\varphi ( 0 ) - \left[ \varphi ( \alpha ) + \varphi' ( \alpha ) ( 0 - \alpha ) \right] \geq \frac{ \varphi'' ( \alpha ) }{ \gamma } ( 0 - \alpha )^2 . \label{eq_local_PB}
\end{equation}
The function $\varphi''$ is continuous, so it takes its minimum $\mu \geq 0$ and maximum $L \geq 0$ on $[ 0, \bar{\alpha} ]$.
The Taylor formula with the integral remainder (see, e.g., \cite{Polyak1987}) gives
\begin{align}
& \varphi ( 0 ) - \left[ \varphi ( \alpha ) + \varphi' ( \alpha ) ( 0 - \alpha ) \right] \notag \\
& \quad = \alpha^2 \int_0^1 \int_0^t \! \varphi'' ( \alpha + \tau ( 0 - \alpha ) ) \, \du \tau \, \du t \geq \frac{\mu}{2} \alpha^2 . \notag
\end{align}
Therefore, the inequality \eqref{eq_local_PB} holds, if $( \mu / 2 ) \geq ( L / \gamma )$. 

We consider two cases: 
\begin{enumerate}
\item If $\mu = 0$, Theorem \ref{thm_PB} implies that $f' ( \rho ) = \kappa I$ for some $\kappa \in \mathbb{R}$. 
Then one can verify $\rho ( \alpha ) = \rho$ for all $\alpha$. 
Therefore, $\Gamma ( \alpha ) = 0$ for all $\alpha$, and the proposition trivially holds with $\gamma = 2$.

\item If $\mu \neq 0$, one can simply choose $\gamma = ( 2 L / \mu ) \geq 2$. 
Write 
\begin{equation}
L := \max \Set{ \varphi'' ( \alpha ; \rho ) | \alpha \in [ 0, \bar{\alpha} ] } . \notag
\end{equation}
Notice that $\varphi'' ( \alpha; \rho )$ is continuous on $[ 0, \bar{\alpha} ] \times \mathcal{D}$ as a function of the pair $( \alpha, \rho )$, and $\mathcal{D}$ is a compact set. 
Therefore, $L$ is continuously dependent on $\rho$ \cite{Berge1963}. 
Similarly, $\mu$ and hence $\gamma$ are also continuously dependent on $\rho$. 
\end{enumerate}
\end{proof}

While the Peierls-Bogoliubov inequality requires $\varphi'' ( \alpha ) \geq 0$ for all $\alpha$, Proposition \ref{prop_local_PB} essentially requires $\varphi'' ( \alpha )$ to be strictly positive restricted on $[ 0, \hat{\alpha} ]$. 
This explains why we call Proposition \ref{prop_local_PB} a local Peierls-Bogoliubov inequality.

\subsection{Proof of Theorem \ref{thm_main}} \label{sec_main_proof}

We present the proofs of the five statements in Theorem \ref{thm_main} one by one. 
The proofs of Statements 1--4 are simple; the difficulties lie in the proof of Statement 5.

\paragraph{Proof of Statement 1}

Statement 1 follows from the following proposition. 

\begin{proposition} \label{prop_finite_step}
For any non-singular density matrix $\rho$ in $\dom (f)$ and $\tau \in ( 0, 1 )$, there exists some $\tilde{\alpha} > 0$ such that
\begin{equation}
f ( \rho ( \alpha ) ) \leq f ( \rho ) + \tau \Braket{ f' ( \rho ), \rho ( \alpha ) - \rho } , \label{eq_finite_step}
\end{equation}
for all $\alpha \in ( 0, \tilde{\alpha} )$.
\end{proposition}

\begin{proof}
Equivalently, we have to verify
\begin{align}
f ( \rho ( \alpha ) ) - \left[ f ( \rho ) + \Braket{ f' ( \rho ), \rho ( \alpha ) - \rho } \right] \leq - (1 - \tau) \Braket{ f' ( \rho ), \rho ( \alpha ) - \rho } . \notag
\end{align}
By Corollary \ref{cor_equiv} and Lemma \ref{lem_epsilon}, it suffices to prove 
\begin{equation}
L D ( \rho ( \alpha ), \rho ) \leq \frac{( 1 - \tau ) D ( \rho ( \alpha ), \rho )}{ \alpha } , \label{eq_equiv_finite_step}
\end{equation}
in a neighborhood of $\rho$, where $L$ denotes the local Lipschitz constant of $f'$ in the neighborhood.
If $\rho$ is a minimizer of $f$ on $\mathcal{D}$, one has $\rho ( \alpha ) = \rho$ by Lemma \ref{lem_fp}; hence the proposition holds.
If $\rho$ is not a minimizer, \eqref{eq_equiv_finite_step} is equivalent to $L \leq ( 1 - \tau ) / \alpha$, which holds when $\alpha$ is small enough.
\end{proof}

\paragraph{Proof of Statement 2}

This is obvious by definition.

\paragraph{Proof of Statement 3}

The Armijo rule ensures that
\begin{equation}
f ( \rho_{k + 1} ) \leq f ( \rho_{k} ) + \tau \Braket{ f' ( \rho_k ), \rho_{k + 1} - \rho_k } , \quad k \in \mathbb{Z}_+ . \notag
\end{equation}
Notice that $\rho_{k + 1} = \rho_k ( \alpha_k )$. 
Statement 3 then follows from Corollary \ref{cor_equiv}.

\paragraph{Proof of Statement 4}

This statement follows from Statement 2 and the compactness of the constraint set $\mathcal{D}$.

\paragraph{Proof of Statement 5}

Equivalently, we will show that any convergent sub-sequence of $( \rho_k )_{k \in \mathbb{N}}$ converges to a minimizer of $f$ on $\mathcal{D}$.

We first check the feasibility of a limit point.

\begin{lemma} \label{lem_domain}
All limit points of $( \rho_k )_{k \in \mathbb{N}}$ lie in $\dom ( f )$.
\end{lemma}

\begin{proof}
Otherwise, Statement 3 in Theorem \ref{thm_main} cannot hold by the continuity of $f$.
\end{proof}

Lemma \ref{lem_domain} allows one to talk about the local Lipschitz constant of $f'$ around any limit point.

\begin{proposition} \label{prop_almost_convergence}
Let $( \rho_k )_{k \in \mathbb{K}}$ be a convergent sub-sequence for some $\mathcal{K} \subseteq \mathbb{N}$, converging to some $\bar{\rho} \in \mathcal{D}$.
Then there exists some constant $\beta > 0$, such that $D ( \rho_k ( \beta ), \rho_k ) \to 0$ as $k \to \infty$ in $\mathcal{K}$.
\end{proposition}

\begin{proof}
If $\rho_{k'}$ is a minimizer for some $k' \in \mathcal{K}$, Lemma \ref{lem_fp} implies that $\rho_k = \rho_{k'}$ for all $k > k'$ in $\mathcal{K}$, and the proposition trivially holds. 
In the rest of the proof, we assume that $\rho_k$ is not a minimizer for all $k \in \mathcal{K}$. 

We will denote by $\gamma_k$ the value of $\gamma$ in Proposition \ref{prop_local_PB} corresponding to $\rho_k$ for all $k$. 
By continuity, $\gamma_k$ converges to some $\bar{\gamma} \geq 2$; hence one has $( 1 / 2 ) \bar{\gamma} \leq \gamma_k \leq 2 \bar{\gamma}$ for large enough $k \in \mathcal{K}$.

\newcommand{\alphal}{\underline{\alpha}}

Suppose that $\liminf \Set{ \alpha_k | k \in \mathcal{K} } \geq \alphal$ for some $\alphal > 0$.
Let $( \alpha_k )_{k \in \mathcal{K}'}$ be a sub-sequence of $( \alpha_k )_{k \in \mathcal{K}}$ converging to $\alphal$.
By assumption, one has $\alpha_k \leq 2 \alphal$ for large enough $k \in \mathcal{K}'$.
Then one can write
\begin{align}
f ( \rho_k ) - f ( \rho_{k + 1} ) 
& \geq - \tau \Braket{ f' ( \rho_k ), \rho_{k + 1} - \rho_k } \notag \\
& \geq \tau \alpha_k^{-1} D ( \rho_{k + 1}, \rho_k ) \notag \\
& = \tau \alpha_k^{\gamma_k - 1} \alpha_k^{- \gamma_k} D ( \rho_{k + 1}, \rho_k ) \notag \\
& \geq \tau \alpha_k^{\gamma_k - 1} \bar{\alpha}^{- \gamma_k} D ( \rho_k ( \bar{\alpha} ), \rho_k ) \notag \\
& \geq C D ( \rho_k ( \bar{\alpha} ), \rho_k ), \notag
\end{align}
where $C := ( 2 \alphal )^{2 \bar{\gamma} - 1}$ is independent of $k$. 
We have applied the definition of the Armijo rule in the first inequality, Corollary \ref{cor_equiv} in the second inequality, and Proposition \ref{prop_local_PB} in the third inequality. 
The proposition follows from the continuity of $f$.

Suppose that $\liminf \Set{ \alpha_k | k \in \mathcal{K} } = 0$. 
Let $( \alpha_k )_{k \in \mathcal{K}'}$ be a sub-sequence of $( \alpha_k )_{k \in \mathcal{K}}$ converging to $0$.
Since then it is impossible to have $\alpha_k = \alpha$ for all $k \in \mathcal{K}'$, one has
\begin{equation}
f ( \rho_k ( r^{-1} \alpha_k ) ) > f ( \rho_k ) + \tau \Braket{ f' ( \rho_k ), \rho_k ( r^{-1} \alpha_k ) } . \notag
\end{equation}
By Lemma \ref{lem_epsilon} and Lemma \ref{cor_equiv} , one can write
\begin{align}
& L D ( \rho_k ( r^{-1} \alpha_k ), \rho_k ) \notag \\
& \quad \geq f ( \rho_k ( r^{-1} \alpha_k ) ) - \left[ f ( \rho_k ) + \Braket{ f' ( \rho_k ), \rho_k ( r^{-1} \alpha_k ) - \rho_k } \right] \notag \\
& \quad > - ( 1 - \tau ) \Braket{ f' ( \rho_k ), \rho_k ( r^{-1} \alpha_k ) } \notag \\
& \quad \geq \frac{ ( 1 - \tau ) D ( \rho_k ( r^{-1} \alpha_k ), \rho_k ) }{ r^{-1} \alpha_k } \notag ,
\end{align}
for large enough $k$ in $\mathcal{K}'$, where $L$ is a local Lipschitz constant of $f'$ in a neighborhood of $\bar{\rho}$.
Proposition \ref{prop_local_PB} then implies
\begin{align}
D ( \rho_k ( r^{-1} \alpha_k ), \rho_k ) \geq \tilde{C} \left[ D ( \rho_k ( r^{-1} \alpha ), \rho_k ) \right]^{1 / \gamma_k} \left[ D ( \rho_k ( \gamma^{-1} \alpha_k ), \rho_k ) \right]^{1 - 1 / \gamma_k}, \notag
\end{align}
where $\tilde{C} := ( 1 - \tau ) / ( r^{-1} \alpha L )$ is independent of $k$.
Since we assume $\rho_k$ is not a minimizer for all $k$, $D ( \rho_k ( r^{-1} \alpha_k ), \rho_k ) \neq 0$ for all $k$.
Then one obtains
\begin{equation}
\left[ D ( \rho_k ( r^{-1} \alpha ), \rho_k ) \right]^{1 / \gamma_k} \leq \tilde{C}^{-1} \left[ D ( \rho_k ( r^{-1} \alpha_k ), \rho_k ) \right]^{1 / \gamma_k} . \notag
\end{equation}
The dependence of $\gamma_k$ on $k$ can be removed by writing
\begin{equation}
\left[ D ( \rho_k ( r^{-1} \alpha ), \rho_k ) \right]^{1 / 2 \bar{\gamma}} \leq \tilde{C}^{-1} \left[ D ( \rho_k ( r^{-1} \alpha_k ), \rho_k ) \right]^{2 / \bar{ \gamma} } , \notag
\end{equation}
for large enough $k \in \mathcal{K}'$.
It remains to show that $D ( \rho_k ( \gamma^{-1} \alpha_k ), \rho_k ) \to 0$ as $k \to \infty$ in $\mathcal{K}'$.
This can be verified by Lemma \ref{lem_D_phi} and the assumption that $\alpha_k \to 0$ as $k \to \infty$ in $\mathcal{K}'$: 
\begin{align}
D ( \rho_k ( r^{-1} \alpha_k ), \rho_k ) & = \varphi_k ( 0 ) - \left[ \varphi_k ( r^{-1} \alpha_k ) ( 0 - r^{-1} \alpha_k ) \right] \notag \\
& \leq \frac{L_k}{2} \left( \frac{\alpha_k}{r} \right) ^ 2 \leq \bar{L} \left( \frac{\alpha_k}{r} \right) ^ 2, \notag
\end{align}
for large enough $k \in \mathcal{K}'$, where $\varphi_k ( t ) := \varphi ( t ; \rho_k )$ for $t \in \mathbb{R}$, $L_k$ denotes the supremum of $\varphi_k''$ on $[ 0, \gamma^{-1} \alpha ]$, and $\bar{L}$ denotes the supremum of $\varphi'' ( \cdot \, ; \bar{\rho} )$ on the same interval. 
We used the fact that $L_k \leq 2 \bar{L}$ for $k$ large enough in the second inequality; notice that $L_k$ converges to $\bar{L}$, as shown at the end of the proof of Proposition \ref{prop_local_PB}.
\end{proof}

If $\bar{\rho}$ is non-singular, Proposition \ref{prop_almost_convergence} implies $D ( \bar{\rho} ( \beta ), \bar{\rho} ) = 0$ for some $\beta > 0$; therefore, $\bar{\rho} ( \beta ) = \bar{\rho}$, so $\bar{\rho}$ is a minimizer by Lemma \ref{lem_fp}. 
However, if $\bar{\rho}$ is singular, $\bar{\rho} ( \beta )$ is not well-defined in \eqref{eq_rho_alpha}. 
Although the equivalent definition of $\bar{\rho} ( \beta )$ given by Lemma \ref{lem_equiv} is still valid when $\bar{\rho}$ is singular, it is unclear whether the limiting argument goes through. 
We show explicitly that Proposition \ref{prop_almost_convergence} implies the optimality of $\bar{\rho}$ in the rest of this sub-section.

The idea is to consider the first-order optimality condition---although $\bar{\rho} ( \beta )$ might be not well-defined when $\bar{\rho}$ is non-singular, the first-order optimality condition is always well-defined.
For any $\rho \in \dom ( f )$, define
\begin{equation}
\psi ( \rho ) := \inf \Set{ \Braket{ f' ( \rho ), \sigma - \rho } | \sigma \in \mathcal{D} } . \label{eq_psi}
\end{equation}
The first-order optimality condition says that a density matrix $\rhostar$ minimizes $f$ on $\mathcal{D}$, if and only if $\psi ( \rhostar ) = 0$ (see, e.g., \cite{Nesterov2004}).
Notice that $\psi$ is a continuous function well-defined on $\dom ( f )$.
Our goal is to show that
\begin{equation}
\psi ( \bar{\rho} ) = \lim_{k \to \infty \in \mathcal{K}} \psi ( \rho_k ) = 0, \notag
\end{equation}
for any convergent sub-sequence $( \rho_k )_{k \in \mathcal{K}}$.

\begin{lemma} \label{lem_psi}
For any non-singular density matrix $\rho$ and $\beta > 0$, it holds that
\begin{equation}
- \beta^{-1} D ( \rho ( \beta ), \rho ) \leq \psi ( \rho ) \leq 0 . \notag
\end{equation}
\end{lemma}

\begin{proof}
The upper bound on $\psi$ is obvious, as one can choose $\sigma = \rho$ in \eqref{eq_psi}. 

It is easily verified that
\begin{equation}
\psi ( \rho ) = \lambda_{\min} ( f' ( \rho ) ) - \Braket{ f' ( \rho ), \rho } , \notag
\end{equation}
for any $\rho \in \mathcal{D}$, where $\lambda_{\min} ( \cdot )$ denotes the minimum eigenvalue.
A direct calculation gives
\begin{equation}
D ( \rho ( \beta ), \rho ) = - \beta \Braket{ f' ( \rho ), \rho ( \beta ) } - \log \tr \exp \left[ \log ( \rho ) - \beta f' ( \rho ) \right] . \notag
\end{equation}
We bound the two terms at the right-hand side separately.
Noticing that $D ( \rho ( \beta ), \rho ) \geq 0$, Corollary \ref{cor_equiv} implies 
\begin{equation}
- \Braket{ f' ( \rho ), \rho ( \beta ) } \geq - \Braket{ f' ( \rho ), \rho }.  \notag 
\end{equation}
As $f' ( \rho ) - \lambda_{\min} ( f' ( \rho ) ) I$ is positive semi-definite, one has
\begin{align}
\log \tr \exp \left\{ \log ( \rho ) - \beta \left[ f' ( \rho ) - \lambda_{\min} ( f' ( \rho ) ) I \right] \right\} \leq \log \tr \exp \log ( \rho ) = 0, \notag
\end{align}
i.e., 
\begin{equation}
- \log \tr \exp \left[ \log ( \rho ) - \beta f' ( \rho ) \right] \geq \beta \lambda_{\min} ( f' ( \rho ) ) . \notag
\end{equation}
The lemma follows.
\end{proof}

Consider any convergent sub-sequence $( \rho_k )_{k \in \mathcal{K}}$ converging to a limit point $\bar{\rho}$.
We have proved that there exists some constant $\beta > 0$ such that $D ( \rho_k ( \beta ), \rho_k ) \to 0$ as $k \to \infty$ in $\mathcal{K}$.
Lemma \ref{lem_psi} and the continuity of $\psi$ then imply
\begin{equation}
\lim_{k \to \infty \text{ in } \mathcal{K}} \psi ( \rho_k ) = \psi( \bar{ \rho } ) = 0, \notag
\end{equation}
which establishes the optimality of $\bar{\rho}$.

\section{Numerical Experiment: ML Quantum State Tomography} \label{sec_exp}


\begin{figure} 
\centering
\includegraphics[width=.8\textwidth]{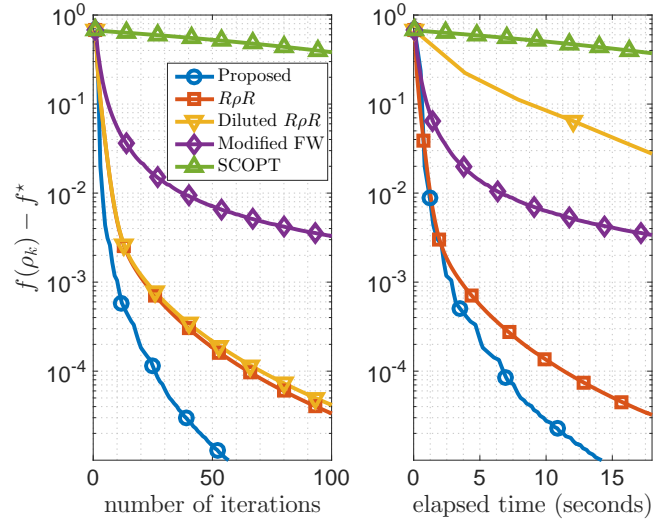} 
\caption{\label{fig_q6} The 6-qubit case.}
\end{figure}

\begin{figure} 
\centering
\includegraphics[width=.8\textwidth]{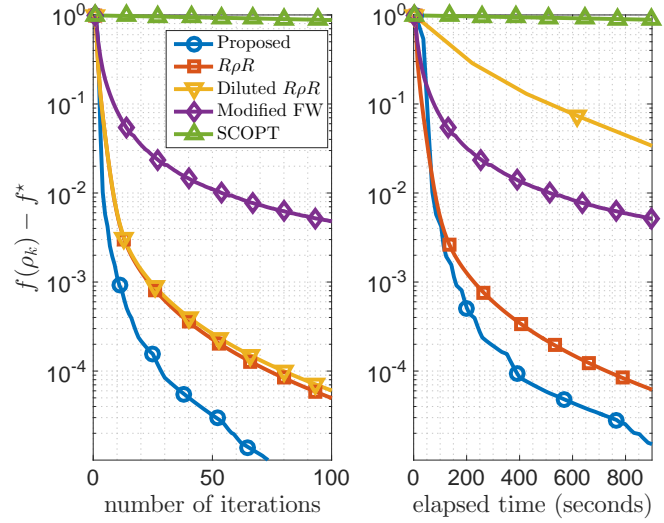} 
\caption{\label{fig_q8} The 8-qubit case.}
\end{figure}

Quantum state tomography is the problem of estimating an unknown density matrix $\rho \in \mathcal{C}^{d \times d}$, by measuring multiple independent and identically prepared copies of it (for details, see, e.g., \cite{Paris2004}).
It is essential in quantum information applications; for example, researchers estimate the density matrix of a prepared quantum gate for calibration.

A measurement setting is mathematically described by a probability operator-valued measure (POVM), a set of Hermitian positive semi-definite matrices summing up to the identity.
Let $\mathcal{M} := \Set{ M_j : j \in \mathcal{J} }$ be a POVM.
The corresponding measurement outcome of $\rho$ is a random variable $\xi$, taking values in $\mathcal{J}$ and satisfying $\mathbb{P} \Set{ \xi = j } = \tr ( M_j \rho )$ for all $j \in \mathcal{J}$.
Given $n$ independent measurement outcomes on $n$ copies, the normalized negative log-likelihood function is then given by $f_3$ (cf. \eqref{eq_f2}), where each $M_i$ is an element in the POVM applied to the $i$-th copy of $\rho$.

The experimental data we have was generated following the setting in \cite{Haffner2005}, in which Pauli-based measurements are used to measure the $W$-state (a specific single-rank density matrix). 
Under this setting, each $M_i$ is a single-rank matrix of the form $v v^H$, $v$ being a tensor product of eigenvectors of Pauli matrices.

As discussed in Section \ref{sec_motivation}, $f_3$ is not Lispchitz continuous in its function value nor its gradient; hence there are few guaranteed-to-converge existing algorithms.
To the best of our knowledge, the diluted $R \rho R$ algorithm \cite{Rehacek2007}, SCOPT \cite{Tran-Dinh2015b}, and the modified Frank-Wolfe algorithm \cite{Odor2016} are the only existing algorithms that are guaranteed to converge.
We will also consider the $R \rho R$ algorithm \cite{Hradil1997}, which does not converge in some cases \cite{Rehacek2007}, but is much faster than its diluted version, the diluted $R \rho R$ algorithm.

We compare the convergence speeds for the $6$-qubit ($d = 2^6$) and $8$-qubit ($d = 2^8$) cases, in Fig. \ref{fig_q6} and \ref{fig_q8}, respectively.
The corresponding ``sample sizes'' (i.e., number of summands in $f_3$) are $n = 60640$ and $n = 460938$, respectively.
The experiments were done in MATLAB R2015b, on a MacBook Pro with an Intel Core i7 2.8GHz processor and 16GB DDR3 memory.
We set $\alpha = 10$, and $\gamma = \tau = 0.5$ in Algorithm \ref{alg} for both cases.
In both figures, $f^\star$ denotes the minimum value of $f_3$ found by the five algorithms in 120 iterations.

One can observe that the EG method with Armijo line search has the fastest empirical convergence speed, in terms of the actual elapsed time.
The numerical results can be explained by theory.
\begin{enumerate}
\item The diluted $R \rho R$ algorithm, using the notation of this paper, iterates as 
\begin{equation}
\rho_{k + 1} = c_k^{-1} \left[ I + \lambda_k f' ( \rho_k ) \right]^H \rho_k \left[ I + \lambda_k f' ( \rho_k ) \right], \notag
\end{equation}
where $c_k$ normalizes the trace of $\rho_{k + 1}$, and to guarantee convergence, the step size $\lambda_k$ is computed by exact line search.
The exact line search procedure renders the algorithm slow.

\item SCOPT is a projected gradient method for minimizing self-concordant functions \cite{Nesterov1994,Nesterov2004}, which chooses the step size such that each iterate lies in the Dikin ellipsoid centered at the previous iterate. 
It is easily checked that $f_3$ is a self-concordant function of parameter $2 \sqrt{ n }$.
Following the theory in \cite{Nesterov1994,Nesterov2004}, the radius of the Dikin ellipsoid shrinks at the rate $O( n^{-1/2} )$, so SCOPT becomes slow when $n$ is large.

\item The modified Frank-Wolfe algorithm is essentially the same as the standard Frank-Wolfe algorithm, with a novel step size to guarantee convergence for minimizing $f_3$. 
Like the standard Frank-Wolfe algorithm, the modified version suffers for a sub-linear convergence rate due to the zig-zagging phenomenon (see, e.g., \cite{Lacoste-Julien2015} for an illustration).
\end{enumerate}

We notice that the empirical convergence rate of the EG method with Armijo line search is linear, despite that $f_3$ is not globally strongly convex.

\section{A Historical Remark}

We have discussed existing analyses of the EG method in Section \ref{sec_intro}. 
As for Armijo line search, there are few existing convergence results as general as Theorem \ref{thm_main}.
The Armijo rule was originally proposed for unconstrained convex minimization \cite{Armijo1966}, assuming that the loss function has a Lipschitz continuous gradient. 
Bertsekas extended the formulation of Armijo line search for continuously differentiable convex functions, and showed that the projected gradient method with Armijo line search (henceforth abbreviated as PGA) always converges for the box and positive orthant constraints in \cite{Bertsekas1976}.
According to \cite{Bertsekas1976} and \cite{Gafni1982}, Goldstein proved the convergence of PGA for a class of constraint sets in a conference paper in 1974.
A general convergence result for PGA, which is valid for any continuously differentiable convex function and any convex constraint set, appeared first in \cite{Gafni1982}, and was then summarized in \cite{Bertsekas2016} (what we cited is the last edition of the book). 
To the best of our knowledge, there was no such general convergence result for the EG method. 
Our Theorem \ref{thm_main} fills this gap.

%

\bibliography{list}
\bibliographystyle{acm}

\end{document}